\documentclass{amsart}
\usepackage{amsmath}
\usepackage{amssymb}
\usepackage{amsthm}
\usepackage{epsfig}
\usepackage{amsfonts}
\usepackage{amscd}
\usepackage{mathrsfs}
\usepackage{enumerate}
\usepackage{latexsym}

\newtheorem{theorem}{Theorem}[section]
\newtheorem{lemma}[theorem]{Lemma}
\newtheorem{corollary}[theorem]{Corollary}
\newtheorem{proposition}[theorem]{Proposition}
\newtheorem{question}[theorem]{Question}
\theoremstyle{definition}
\theoremstyle{notation}

\newtheorem{example}[theorem]{Example}

\theoremstyle{remark}
\newtheorem{remark}[theorem]{Remark}
\numberwithin{equation}{section}

\begin{document}

\title[Embedding measure spaces]{Embedding measure spaces}

\author[M.R. Koushesh]{M.R. Koushesh}

\address{Department of Mathematical Sciences, Isfahan University of Technology, Isfahan 84156-83111, Iran.}
\address{School of Mathematics, Institute for Research in Fundamental Sciences (IPM), P.O. Box: 19395-5746, Tehran, Iran}
\email{koushesh@cc.iut.ac.ir}

\thanks{This research was in part supported by a grant from IPM (No. 90030052).}

\subjclass[2010]{54D35, 54D60, 28A05, 28C15.}

\keywords{Ultrafilter, thick subset, set of full outer measure, topological measure space, Baire measure, Stone--\v{C}ech compactification, realcompactification}

\begin{abstract}
For a given measure space $(X,{\mathscr B},\mu)$ we construct all measure spaces $(Y,{\mathscr C},\lambda)$ in which  $(X,{\mathscr B},\mu)$ is embeddable. The construction is modeled on the ultrafilter construction of the Stone--\v{C}ech compactification of a completely regular topological space. Under certain conditions the construction simplifies. Examples are given when this simplification occurs.
\end{abstract}

\maketitle

\section{Introduction}

A measurable  space  $(X,{\mathscr B})$ is said to be  {\em embedded} in a measurable  space $(Y,{\mathscr C})$ (denoted by $(X,{\mathscr B})\subseteq(Y,{\mathscr C})$)  if  $X\subseteq Y$ and
\[{\mathscr B}=\{C\cap X:C\in{\mathscr C}\}.\]
A measure space $(X,{\mathscr B}, \mu)$ is said to be  {\em embedded} in a measure space $(Y,{\mathscr C},\lambda)$ (denoted by $(X,{\mathscr B}, \mu)\subseteq(Y,{\mathscr C},\lambda)$) if  $(X,{\mathscr B})\subseteq(Y,{\mathscr C})$ and $\mu(C\cap X)=\lambda(C)$ for each $C\in {\mathscr C}$. In this note,  for a given measure space $(X,{\mathscr B}, \mu)$, we construct all measure spaces $(Y,{\mathscr C}, \lambda)$ in which  $(X,{\mathscr B}, \mu)$ is embedded. Equivalently, for a given measure space $(X,{\mathscr B}, \mu)$, we construct all measure spaces $(Y,{\mathscr C}, \lambda)$ which contains $(X,{\mathscr B}, \mu)$ as a thick subspace. (Recall that a subset $X$ of a measure space $(Y,{\mathscr C}, \lambda)$ is said to be {\em thick} (or {\em of full outer measure}) if $\lambda(C)=0$ for each $C\in {\mathscr C}$ such that $C\subseteq Y\backslash X$, equivalently,  if $\lambda_*(Y\backslash X)=0$, where $\lambda_*$ denotes the inner measure induced by $\lambda$. If $X$  is a thick subset of $(Y,{\mathscr C}, \lambda)$ and if
\[{\mathscr B}=\{C\cap X:C\in {\mathscr C}\}\]
and $\mu(C\cap X)=\lambda(C)$ for each $C\in {\mathscr C}$, then $(X,{\mathscr B}, \mu)$ is a measure space which is embedded in $(Y,{\mathscr C}, \lambda)$. Conversely, if $(X,{\mathscr B}, \mu)$ is embedded in $(Y,{\mathscr C}, \lambda)$, then $X$ is a thick subset of $(Y,{\mathscr C}, \lambda)$; see e.g. Theorem 17.A of \cite{H}.) Our construction here is analogous to the ultrafilter construction of the Stone--\v{C}ech compactification of a completely regular topological space $X$. (Completely regular topological spaces are always assumed to be Hausdorff.)

We recall some basic facts, definitions and notation. For details we refer the reader to \cite{E}, \cite{GJ} and \cite{H}.  Let  $(X,{\mathscr B})$ be  a measurable  space. A non-empty ${\mathscr A}\subseteq{\mathscr B}$ is called a {\em filter-base}  in  ${\mathscr B}$ if for every $A,B\in {\mathscr A}$ there exists a non-empty $C\in  {\mathscr A}$ such that $C\subseteq A\cap B$. A {\em filter}  ${\mathscr F}$  in  ${\mathscr B}$ is a filter-base such that $B\in{\mathscr F}$ whenever $B\in{\mathscr B}$ and $F\subseteq B$ for some $F\in {\mathscr F}$. An {\em ultrafilter}  in  ${\mathscr B}$ is a maximal (with respect to $\subseteq$) filter. An ultrafilter is called {\em free} if it has empty intersection, otherwise, it is called {\em fixed}. An ultrafilter is said to have  {\em the countable intersection property} ({\em c.i.p.}, in short) if every countable number  of its elements has a non-empty intersection.
It is known that every filter-base in ${\mathscr B}$ is contained in some ultrafilter  in  ${\mathscr B}$, and that a filter-base   ${\mathscr A}$ in  ${\mathscr B}$ is an ultrafilter if and only if for each $B\in {\mathscr B}$  if $B$ meets every element of  ${\mathscr A}$ then $B\in  {\mathscr A}$. Note that an ultrafilter  ${\mathscr F}$ in  ${\mathscr B}$ has c.i.p. if and only if it is {\em $\sigma$-complete}, i.e.,  it is closed under countable intersections.

Let $X$ be a topological space. By a {\em zero-set} in $X$ we mean a set of the form $f^{-1}(0)$ where $f:X\rightarrow [0,1]$ is continuous; the complement of a zero-set is called a {\em cozero-set}; denote $\mathrm{Z}(f)= f^{-1}(0)$ and $\mathrm{Coz}(f)=X\backslash\mathrm{Z}(f)$. Let ${\mathscr Z}(X)$ and $\mathrm{Coz}(X)$ denote the set of all zero-sets and the set of all cozero-sets of $X$, respectively. Let $X$ be a completely regular topological  space.  A {\em compactification} of $X$ is a compact Hausdorff topological space which contains $X$ as a dense subspace. We denote by $\beta X$ the {\em Stone--\v{C}ech compactification} of $X$, which always exists, and is characterized by either of the following properties:
\begin{itemize}
\item Every continuous function from $X$ to a compact space is continuously extendible over $\beta X$.
\item Every continuous function from $X$ to $[0,1]$ is continuously extendible over $\beta X$.
\item For every $Z,S\in {\mathscr Z}(X)$ such that $Z\cap S=\emptyset$ we have
\[\mathrm{cl}_{\beta X}Z\cap\mathrm{cl}_{\beta X}S=\emptyset.\]
\item For every $Z,S\in {\mathscr Z}(X)$ we have
\[\mathrm{cl}_{\beta X}(Z\cap S)=\mathrm{cl}_{\beta X}Z\cap\mathrm{cl}_{\beta X}S.\]
\end{itemize}
Note, in particular, this implies that disjoint zero-sets (and thus disjoint closed-open subsets)  in $X$  have disjoint closures in $\beta X$. For a completely regular topological space $X$ the {\em Hewitt realcompactification} $\upsilon X$ of $X$ is the subspace of $\beta X$ defined by
\[\upsilon X=\bigcap\big\{C:C\in\mathrm{Coz}(\beta X)\mbox{ and }X\subseteq C\big\}.\]
A topological space is said to be {\em realcompact} if it is homeomorphic to a closed subspace of some topological  product of the real line. Every regular Lindel\"{o}f topological space is realcompact.  It is known that a completely regular topological space $X$  is realcompact if and only if $X=\upsilon X$ if and only if for every $p\in \beta X\backslash X$ there exists a zero-set $Z$ in $\beta X$ such that $p\in Z$ and $Z\cap X=\emptyset$.

A {\em topological measurable space} is a triple $(X,{\mathscr O},{\mathscr B})$ where $(X,{\mathscr B})$  is a measurable  space and  $(X,{\mathscr O})$ is a topological space such that ${\mathscr O}\subseteq {\mathscr B}$, i.e., every open set (and thus every Borel set) is measurable.

This note is organized as follows. In Section 2 we construct all measure spaces $(Y,{\mathscr C},\lambda)$ in which a given  measure space $(X,{\mathscr B},\mu)$ is embedded.  In Section 3 we simplify the construction under certain additional conditions on $(X,{\mathscr B},\mu)$. Indeed, we prove that if the points of $X$ are separated by measurable sets in ${\mathscr B}$ and there is no free ultrafilter in ${\mathscr B}$ with c.i.p., then $(X,{\mathscr B}, \mu)$ is embeddable in   $(Y,{\mathscr C}, \lambda)$ if and only if  $(Y,{\mathscr C}, \lambda)$ is obtained from  $(X,{\mathscr B}, \mu)$ by ``blowing" certain points of $X$ up and  ``pasting"  a certain measurable space to $X$ in a certain way. In Section 4 we provide examples satisfying the assumption of the theorem in Section 3, i.e., we find examples of measure spaces $(X,{\mathscr B}, \mu)$ with no free ultrafilter in  ${\mathscr B}$ having  c.i.p. It turns out that the class of such measure spaces  $(X,{\mathscr B}, \mu)$  is reasonably large (e.g., it contains the class of all first-countable realcompact topological measure spaces, thus in particular, containing all $n$-dimensional  Lebesgue  measure spaces) and behaves very nicely in connection with the standard operations on measure spaces (e.g., we show that for any $\sigma$-finite measure spaces  $(X, {\mathscr B}, \mu)$ and $(Y, {\mathscr C}, \lambda)$ such that in each of which singletons are measurable, considering the measure space $(X\times Y, {\mathscr B}\times{\mathscr C}, \mu\times\lambda)$, there is no free ultrafilter  in ${\mathscr B}\times{\mathscr C}$  with  c.i.p. if and only if  there is a free ultrafilter  with c.i.p. neither in ${\mathscr B}$ nor in ${\mathscr C}$.) Finally, in Section 5 we give examples of measure spaces  $(X, {\mathscr B}, \mu)$  having arbitrarily large number of free ultrafilter  in ${\mathscr B}$  with  c.i.p. We leave some problems open which are formally stated.

\section{The construction of measure spaces in which a given  measure space $(X,{\mathscr B},\mu)$ is embeddable}

The following lemma is well known.

\begin{lemma}\label{LJHG}
Let $(X,{\mathscr B})$ be a measurable space. Let ${\mathscr U}$ be an ultrafilter in ${\mathscr B}$.
\begin{itemize}
\item[\rm(1)] For any $B\in{\mathscr B}$ either $B\in{\mathscr U}$ or $X\backslash B\in{\mathscr U}$.
\item[\rm(2)] Suppose that ${\mathscr U}$ has c.i.p. and $B_1,B_2,\ldots\in{\mathscr B}$. Then
\begin{equation}\label{GGUJ}
\bigcup_{n=1}^\infty B_n\in {\mathscr U}
\end{equation}
if and only if $B_n\in {\mathscr U}$ for some $n\in \mathbb{N}$.
\end{itemize}
\end{lemma}

\begin{proof}
To show (1), note that if $B\notin{\mathscr U}$ for some $B\in{\mathscr B}$, then, since ${\mathscr U}$ is an ultrafilter, we have $B\cap U=\emptyset$ for some $U\in{\mathscr U}$. Thus $U\subseteq X\backslash B$, which implies that $X\backslash B\in {\mathscr U}$.

To show (2), note that (\ref{GGUJ}) holds trivially if $B_n\in {\mathscr U}$ for some $n\in\mathbb{N}$. To show the converse, suppose that (\ref{GGUJ}) holds, while  $B_n\notin {\mathscr U}$ for each $n\in\mathbb{N}$. For each $n\in\mathbb{N}$ (since ${\mathscr U}$ is an ultrafilter) there exists some $U_n\in {\mathscr U}$ such that $B_n\cap U_n=\emptyset$.  Now
\[\bigcap_{i=1}^\infty U_i\cap\bigcup_{n=1}^\infty B_n=\bigcup_{n=1}^\infty\Big(\bigcap_{i=1}^\infty U_i\cap B_n\Big)\subseteq\bigcup_{n=1}^\infty(U_n\cap B_n)=\emptyset\]
contradicting  the fact that ${\mathscr U}$ has c.i.p.
\end{proof}

\begin{theorem}\label{FDDDB}
Let $(X,{\mathscr B},\mu)$ be a measure space. Then $(Y,{\mathscr C},\lambda)$ is a measure space in which $(X,{\mathscr B},\mu)$ is embedded if and only if
there exists a measurable space $(Z,{\mathscr D})$, a collection $\{{\mathscr D}_B:B\in{\mathscr B}\}$ of non-empty subsets of ${\mathscr D}$ such that
\begin{itemize}
\item[\rm(1)] $\emptyset\in{\mathscr D}_\emptyset$;
\item[\rm(2)] if $B\in{\mathscr B}$, then
\[\{Z\backslash D:D\in {\mathscr D}_B\}\subseteq{\mathscr D}_{X\backslash B};\]
\item[\rm(3)] if $B_1,B_2,\ldots\in {\mathscr B}$, then
\[\Big\{\bigcup_{n=1}^\infty D_n :D_n\in {\mathscr D}_{B_n}\Big\}\subseteq{\mathscr D}_{\;\bigcup_{n=1}^\infty B_n};\]
\end{itemize}
and a collection $\{S_{\mathscr U}:{\mathscr U}\in {\mathbb U}\}$ of pairwise disjoint non-empty sets, bijectively indexed by a collection ${\mathbb U}$ of ultrafilters in ${\mathscr B}$ with c.i.p., where the sets $X$, $S_{\mathscr U}$ for any ${\mathscr U}\in {\mathbb U}$, and $Z$ are pairwise disjoint, such that
\[Y=X\cup\bigcup_{{\mathscr U}\in{\mathbb U}} S_{\mathscr U}\cup Z,\]
\[{\mathscr C}=\Big\{B\cup \bigcup_{B\in {\mathscr U}\in {\mathbb U}} S_{\mathscr U}\cup D:B\in{\mathscr B}\mbox{ and }D\in{\mathscr D}_B\Big\},\]
and $\lambda:{\mathscr C}\rightarrow [0,\infty]$ is given by
\[\lambda\Big(B\cup\bigcup_{B\in{\mathscr U}\in {\mathbb U}} S_{\mathscr U}\cup D\Big)=\mu (B)\]
for each $B\in {\mathscr B}$ and $D\in {\mathscr D}_B$.
\end{theorem}

\begin{proof}
Suppose that $Y$, ${\mathscr C}$ and $\lambda$ are defined as in the statement of the theorem.  We show that $(Y,{\mathscr C}, \lambda)$ is  a measure space in which $(X,{\mathscr B}, \mu)$ is embedded. First, we verify that ${\mathscr C}$ is a $\sigma$-algebra on $Y$. By (1) we have $\emptyset\in{\mathscr C}$. Let $C\in{\mathscr C}$. We show that $Y\backslash C\in{\mathscr C}$. Let
\[C=B\cup \bigcup_{B\in {\mathscr U}\in {\mathbb U}} S_{\mathscr U}\cup D\]
for some  $B\in {\mathscr B}$ and $D\in {\mathscr D}_B$. Note that for each ${\mathscr U}\in {\mathbb U}$ we have $B\notin {\mathscr U}$ if and only if  $X\backslash B\in {\mathscr U}$; this is because if $B\notin {\mathscr U}$ then $X\backslash B\in {\mathscr U}$ by Lemma \ref{LJHG}; the converse is trivial. Therefore
\begin{eqnarray*}
Y\backslash C&=&\Big(X\cup \bigcup_{{\mathscr U}\in {\mathbb U}} S_{\mathscr U}\cup Z\Big)\backslash\Big(B\cup \bigcup_{B\in {\mathscr U}\in {\mathbb U}} S_{\mathscr U}\cup D\Big)\\&=&(X\backslash B)\cup\Big(\bigcup_{{\mathscr U}\in {\mathbb U}} S_{\mathscr U}\backslash\bigcup_{B\in {\mathscr U}\in {\mathbb U}} S_{\mathscr U}\Big)\cup (Z\backslash D)\\&=&(X\backslash B)\cup\bigcup_{B\notin {\mathscr U}\in {\mathbb U}} S_{\mathscr U}\cup (Z\backslash D)\\&=&(X\backslash B)\cup\bigcup_{X\backslash B\in {\mathscr U}\in {\mathbb U}} S_{\mathscr U}\cup (Z\backslash D).
\end{eqnarray*}
By (2) we have $Z\backslash D\in {\mathscr D}_{X\backslash B}$. Thus $Y\backslash C\in {\mathscr C}$. Now, to show that ${\mathscr C}$ is closed under countable unions, let $C_1,C_2,\ldots\in{\mathscr C}$. Then
\[C_n=B_n\cup \bigcup_{B_n\in {\mathscr U}\in {\mathbb U}} S_{\mathscr U}\cup D_n\]
where $B_n\in {\mathscr B}$ and $D_n\in {\mathscr D}_{B_n}$ for each $n\in\mathbb{N}$. Using Lemma \ref{LJHG}, we have
\begin{eqnarray*}
\bigcup_{n=1}^\infty C_n&=&\bigcup_{n=1}^\infty \Big(B_n\cup \bigcup_{B_n\in {\mathscr U}\in {\mathbb U}} S_{\mathscr U}\cup D_n\Big)\\&=&\bigcup_{n=1}^\infty B_n\cup \bigcup_{n=1}^\infty \bigcup_{B_n\in {\mathscr U}\in {\mathbb U}} S_{\mathscr U}\cup \bigcup_{n=1}^\infty D_n\\&=&\bigcup_{n=1}^\infty B_n\cup \bigcup_{\bigcup_{n=1}^\infty B_n\in {\mathscr U}\in {\mathbb U}} S_{\mathscr U}\cup \bigcup_{n=1}^\infty D_n.
\end{eqnarray*}
By (3) we have
\[\bigcup_{n=1}^\infty D_n\in {\mathscr D}_{\;\bigcup_{n=1}^\infty B_n}.\]
Thus
\[\bigcup_{n=1}^\infty C_n\in {\mathscr C}.\]
This shows that ${\mathscr C}$ is a $\sigma$-algebra on $Y$. Next, we show that $\lambda$ is a measure on ${\mathscr C}$. Note that $\lambda(\emptyset)=0$.  If $C_1,C_2,\ldots\in{\mathscr C}$ are disjoint, then using the above results and notation we have
\[\lambda\Big(\bigcup_{n=1}^\infty C_n\Big)=\mu\Big(\bigcup_{n=1}^\infty B_n\Big)=\sum_{n=1}^\infty\mu(B_n)=\sum_{n=1}^\infty\lambda(C_n).\]
This shows that  $(Y,{\mathscr C}, \lambda)$ is  a measure space.  Now we show that $(X,{\mathscr B}, \mu)$ is embedded in  $(Y,{\mathscr C}, \lambda)$. Obviously, by our  definitions we have  $X\subseteq Y$ and $C\cap X\in {\mathscr B}$ for each $C\in {\mathscr C}$. Conversely, for each $B\in {\mathscr B}$, since by our assumption ${\mathscr D}_B$ is non-empty, we have  $B=C\cap X$ for some $C\in {\mathscr C}$. Thus
\[{\mathscr B}=\{C\cap X:C\in {\mathscr C}\}.\]
Also, it is obvious that $\lambda(C)=\mu(C\cap X)$ for each $C\in {\mathscr C}$. Therefore $(X,{\mathscr B}, \mu)$ is embedded in  $(Y,{\mathscr C}, \lambda)$.

Now, suppose that $(Y,{\mathscr C}, \lambda)$ is a measure space in which $(X,{\mathscr B}, \mu)$ is embedded. We show that  $(Y,{\mathscr C}, \lambda)$ can be constructed as in the previous part.  Note that $X\subseteq Y$. Define
\[Z=\{p\in Y\backslash X:p\in C\subseteq Y\backslash X\mbox{ for some }C\in {\mathscr C} \}\]
and
\[ {\mathscr D}=\{C\cap Z:C\in {\mathscr C}\}.\]
Then obviously $(Z,{\mathscr D})$ is a measurable space.  Define
\[{\mathscr D}_B=\{C\cap Z:C\in {\mathscr C}\mbox{ and }C\cap X=B\}\]
for each $B\in {\mathscr B}$. Obviously ${\mathscr D}_B\subseteq {\mathscr D}$ and ${\mathscr D}_B$ is non-empty for each $B\in {\mathscr B}$. We verify that conditions (1)--(3) of the theorem hold. Condition (1) holds trivially. To show condition (2), note that if $D\in {\mathscr D}_B$ for some $B\in {\mathscr B}$ then $D=C\cap Z$, where $C\in {\mathscr C}$ and $C\cap X=B$. Thus
\[Z\backslash D=Z \backslash (C\cap Z)=Z\cap (Y\backslash C).\]
Now, since
\[(Y\backslash C)\cap X=X\backslash (C\cap X)=X\backslash B\]
we have $Z\backslash D\in {\mathscr D}_{X\backslash B}$. Therefore
\[\{Z\backslash D:D\in {\mathscr D}_B\}\subseteq{\mathscr D}_{X\backslash B}.\]
To show condition (3), let $B_n\in{\mathscr B}$ and $D_n\in{\mathscr D}_{B_n}$ for each $n\in\mathbb{N}$. Then $D_n=C_n\cap Z$ where $C_n\in {\mathscr C}$ and $C_n\cap X=B_n$ for each $n\in\mathbb{N}$. We have
\[\bigcup_{n=1}^\infty D_n=\bigcup_{n=1}^\infty (Z\cap C_n)=Z\cap\bigcup_{n=1}^\infty C_n\]
and
\[X\cap\bigcup_{n=1}^\infty C_n =\bigcup_{n=1}^\infty (X\cap C_n)=\bigcup_{n=1}^\infty B_n\]
where
\[\bigcup_{n=1}^\infty C_n\in {\mathscr C}.\]
Thus
\[\bigcup_{n=1}^\infty D_n\in {\mathscr D}_{\;\bigcup_{n=1}^\infty B_n},\]
i.e.,
\[\Big\{\bigcup_{n=1}^\infty D_n :D_n\in {\mathscr D}_{B_n}\Big\}\subseteq{\mathscr D}_{\;\bigcup_{n=1}^\infty B_n}.\]
This shows conditions (1)--(3). For each $p\in (Y\backslash X)\backslash Z$ let
\[ {\mathscr U}_p=\{C\cap X:p\in C\in {\mathscr C}\}.\]

\medskip

\noindent {\bf Claim 1.} {\em For each  $p\in (Y\backslash X)\backslash Z$ the set ${\mathscr U}_p$ is an ultrafilter in ${\mathscr B}$ which has c.i.p.}

\medskip

\noindent {\em Proof of Claim 1.} First note that
\begin{equation}\label{UY}
(Y\backslash X)\backslash Z=\{y\in Y\backslash X:C\cap X\neq\emptyset\mbox{ for each } C\in {\mathscr C} \mbox{ with }y\in C \}.
\end{equation}
Let $p\in(Y\backslash X)\backslash Z$. By (\ref{UY}) we have $\emptyset\notin {\mathscr U}_p$. It is obvious that $\emptyset\neq{\mathscr U}_p\subseteq {\mathscr B}$ and that ${\mathscr U}_p$ is closed under finite intersections. Now, suppose that $U\subseteq B$ for some $U\in  {\mathscr U}_p$ and $B\in {\mathscr B}$. Then $U=C\cap X$ for some $C\in  {\mathscr C}$ such that $p\in C$, and $B=G\cap X$ for some $G\in  {\mathscr C}$. If $p\notin G$ then $p\in Y\backslash G\in  {\mathscr C}$. Thus $p\in C\cap (Y\backslash G)\in {\mathscr C}$ and therefore by (\ref{UY}) and the choice of $p$ the set $C\cap (Y\backslash G)\cap X$ is non-empty. But this is a contradiction, as
\[C\cap X\cap(Y\backslash G)=U\cap(Y\backslash G)\subseteq B\cap(Y\backslash G)=(G\cap X)\cap(Y\backslash G)=\emptyset.\]
Thus $p\in G$ and therefore $B=G\cap X\in {\mathscr U}_p$. This  shows that ${\mathscr U}_p$ is a filter in ${\mathscr B}$.  To show that ${\mathscr U}_p$ is an ultrafilter, let $B\in {\mathscr B}$ be such that $B\cap U$ is non-empty for each $U\in {\mathscr U}_p$. Let $B=C\cap X$ for some  $C\in {\mathscr C}$. If $p\notin C$ then $p\in Y\backslash C$ and thus $(Y\backslash C)\cap X\in{\mathscr U}_p$, which is not possible, as  $(Y\backslash C)\cap X$ misses $B$. Therefore $p\in C$ and thus $B=C\cap X\in {\mathscr U}_p$. To show that ${\mathscr U}_p$ has  c.i.p., let  $U_1,U_2,\ldots\in {\mathscr U}_p$. Then $U_n=C_n\cap X$ where $p\in C_n\in  {\mathscr C}$ for each $n\in\mathbb{N}$. Then
\[p\in \bigcap_{n=1}^\infty C_n\in{\mathscr C}\]
and thus
\[\bigcap_{n=1}^\infty U_n=\bigcap_{n=1}^\infty (X\cap C_n)=X\cap\bigcap_{n=1}^\infty C_n\in{\mathscr U}_p.\]
Therefore
\[\bigcap_{n=1}^\infty U_n\neq\emptyset.\]
This proves the claim.

\medskip

\noindent Let
\[{\mathbb U}=\big\{{\mathscr U}_p:p\in (Y\backslash X)\backslash Z\big\}.\]
Then ${\mathbb U}$ is a collection of ultrafilters in   ${\mathscr B}$ with c.i.p. For each ${\mathscr U}\in {\mathbb U}$ define
\[S_{\mathscr U}=\big\{p\in(Y\backslash X)\backslash Z:{\mathscr U}_p={\mathscr U}\big\}.\]
Note that $S_{\mathscr U}$, for each ${\mathscr U}\in {\mathbb U}$, is non-empty, as ${\mathscr U}={\mathscr U}_p$ for some $p\in (Y\backslash X)\backslash Z$ and thus $p\in S_{\mathscr U}$. Also, for any distinct ${\mathscr U},{\mathscr V}\in {\mathbb U}$ we have $S_{\mathscr U}\cap S_{\mathscr V}=\emptyset$, as $p\in S_{\mathscr U}\cap S_{\mathscr V}$ implies that ${\mathscr U}={\mathscr U}_p={\mathscr V}$. Thus
\[\{S_{\mathscr U}:{\mathscr U}\in {\mathbb U}\}\]
is a  bijectively indexed  collection of pairwise disjoint non-empty sets. Note that by our definitions the sets  $X$, $S_{\mathscr U}$ where ${\mathscr U}\in {\mathbb U}$ and $Z$ are pairwise disjoint. Let
\[Y'=X\cup \bigcup_{{\mathscr U}\in {\mathbb U}} S_{\mathscr U}\cup Z \]
\[{\mathscr C}'=\Big\{B\cup \bigcup_{B\in {\mathscr U}\in {\mathbb U}} S_{\mathscr U}\cup D:B\in {\mathscr B}\mbox{ and } D\in {\mathscr D}_B\Big\}\]
and $\lambda':{\mathscr C}'\rightarrow [0,\infty]$ be  given  by
\[\lambda'\Big(B\cup \bigcup_{B\in {\mathscr U}\in {\mathbb U}} S_{\mathscr U}\cup D\Big)=\mu (B)\]
where $B\in {\mathscr B}$ and $D\in {\mathscr D}_B$. By the first part we know that   $(Y',{\mathscr C}', \lambda')$ is  a measure space in which $(X,{\mathscr B}, \mu)$ is embedded. We verify that
\[(Y,{\mathscr C}, \lambda)=(Y',{\mathscr C}', \lambda').\]
By our definition it is obvious that $Y'\subseteq Y$. To show the reverse inclusion, let $p\in Y$. If either $p\in X$ or $p\in Z$ then $p\in Y'$. If $p\in (Y\backslash X)\backslash Z$, then since ${\mathscr U}_p\in  {\mathbb U}$ and by our definition $p\in S_{{\mathscr U}_p}$, we have $p\in Y'$. Thus $Y\subseteq Y'$ and therefore $Y=Y'$.  Next, we verify that ${\mathscr C}={\mathscr C}'$.

\medskip

\noindent {\bf Claim 2.} {\em Let $C\in {\mathscr C}$ and $B=C\cap X$. Then
\[ \bigcup_{B\in {\mathscr U}\in {\mathbb U}}S_{\mathscr U}=C\cap\big((Y\backslash X)\backslash Z\big).\]}

\medskip

\noindent {\em Proof of Claim 2.} Suppose that $p\in S_{\mathscr U}$ for some ${\mathscr U}\in {\mathbb U}$ such that $B\in {\mathscr U}$. If $p\notin C$ then $p\in Y\backslash C\in {\mathscr C}$ and thus
\[(Y\backslash C)\cap X\in {\mathscr U}_p={\mathscr U}.\]
But this is not possible, as $(Y\backslash C)\cap X$ misses $C\cap X=B$. Thus $p\in C$. Also, since $ S_{\mathscr U}\subseteq(Y\backslash X)\backslash Z$ it is obvious that $p\in (Y\backslash X)\backslash Z$. To show the reverse inclusion, note that for each $p\in C\cap((Y\backslash X)\backslash Z)$ since  $p\in C$ we have $B=C\cap X\in{\mathscr U}_p$ and $p\in S_{{\mathscr U}_p}$. This proves the claim.

\medskip

\noindent Now, let $C'\in  {\mathscr C}'$. Then
\[C'=B\cup \bigcup_{B\in {\mathscr U}\in {\mathbb U}} S_{\mathscr U}\cup D.\]
for some $B\in {\mathscr B}$ and $D\in {\mathscr D}_B$. Thus, by the way we have defined ${\mathscr D}_B$ we have $D=C\cap Z$, for some $C\in {\mathscr C}$ such that $C\cap X=B$. By Claim 2 we have
\[C'=B\cup \bigcup_{B\in {\mathscr U}\in {\mathbb U}} S_{\mathscr U}\cup D=(C\cap X)\cup \big(C\cap\big((Y\backslash X)\backslash Z\big)\big)\cup(C\cap Z)=C\in {\mathscr C}\]
Therefore ${\mathscr C}'\subseteq{\mathscr C}$. To show the reverse inclusion let $C\in {\mathscr C}$. Let $B=C\cap X\in {\mathscr B}$ and $D=C\cap Z$. Then $D\in{\mathscr D}_B$, by the way we have defined ${\mathscr D}_B$, and thus by Claim 2 we have
\[C=(C\cap X)\cup\big(C\cap \big((Y\backslash X)\backslash Z\big)\big)\cup(C\cap Z)=B\cup \bigcup_{B\in {\mathscr U}\in {\mathbb U}} S_{\mathscr U}\cup D\in {\mathscr C}'.\]
Therefore ${\mathscr C}\subseteq{\mathscr C}'$, which together with the above shows that ${\mathscr C}={\mathscr C}'$. The fact that $\lambda=\lambda'$ is trivial, as by the above for each $C\in {\mathscr C}$ we have
\[\lambda(C)=\mu(C\cap X)=\lambda'(C).\]
This completes the proof.
\end{proof}

Let $(Y,{\mathscr C}, \lambda)$ be a measure space in which $(X,{\mathscr B},\mu)$ is embedded. Assume the representation and notation given in  Theorem \ref{FDDDB}. Then
\[Y=X\cup \bigcup_{{\mathscr U}\in {\mathbb U}} S_{\mathscr U}\cup Z \]
where (by the proof of Theorem \ref{FDDDB})
\[Z=\{p\in Y\backslash X:p\in C\subseteq Y\backslash X\mbox{ for some }C\in {\mathscr C} \}.\]
Thus
\begin{eqnarray*}
\bigcup_{{\mathscr U}\in {\mathbb U}} S_{\mathscr U}&=&(Y\backslash X)\backslash Z\\&=&\{p\in Y\backslash X:C\cap X\neq\emptyset \mbox{ for each }C\in {\mathscr C} \mbox{ with }p\in C \}.
\end{eqnarray*}
We verify that
\begin{eqnarray*}
&&\bigcup_{{\mathscr U}\in {\mathbb U} \mbox{ is free } } S_{\mathscr U}=\\&&\big\{p\in (Y\backslash X)\backslash Z:p \mbox{ is separated from each }x\in X \mbox{ by  sets in }{\mathscr C}\big\}
\end{eqnarray*}
and consequently
\begin{eqnarray*}
&&\bigcup_{{\mathscr U}\in {\mathbb U} \mbox{ is fixed} } S_{\mathscr U}=\\&&\big\{p\in(Y\backslash X)\backslash Z:p \mbox{ is not separated from some }x\in X \mbox{ by  sets in }{\mathscr C}\big\}.
\end{eqnarray*}
To show this, let $p\in S_{{\mathscr U}}$ for some free ${\mathscr U}\in {\mathbb U}$. Let $x\in X$. Since ${\mathscr U}$ is free, we have $x\notin U$ for some $U\in {\mathscr U}$. Let $D\in {\mathscr D}_U$. Then
\[C=U\cup \bigcup_{U\in {\mathscr V}\in {\mathbb U}} S_{\mathscr V}\cup D \in {\mathscr C}\]
is such that $p\in C$ and $x\notin C$. Conversely, let $p\in (Y\backslash X)\backslash Z$ be such that it can be separated from each $x\in X$ by a measurable set in ${\mathscr C}$. Let ${\mathscr U}\in {\mathbb U}$ be such that $p\in S_{\mathscr U}$. Suppose that ${\mathscr U}$ is not free. Let $x\in \bigcap{\mathscr U}$. Let
\[C=B\cup \bigcup_{B\in {\mathscr V}\in {\mathbb U}} S_{\mathscr V}\cup D \in {\mathscr C},\]
where $B\in {\mathscr B}$ and $D\in {\mathscr D}_B$, be such that $p\in C$ and $x\notin C$. Then, since $p\in S_{\mathscr U}$, we have $B\in {\mathscr U}$, and thus $x\in B$, which is not possible,  as $B\subseteq C$. Therefore  ${\mathscr U}$ is free.

Thus, in the absence of free ultrafilters in ${\mathscr B}$, each $p\in Y\backslash X$ (depending on whether $p\in Z$ or $p\notin Z$) either ``separates" from the whole $X$ by a (null) set in ${\mathscr C}$, or tightly ``sticks" to some point $x$ of $X$ so that it cannot be separated from $x$ by any measurable set in  ${\mathscr C}$. In the next section we restrict our attention to measure spaces  $(X,{\mathscr B},\mu)$ having no free ultrafilter in ${\mathscr B}$ with c.i.p. As we will see, this assumption considerably simplifies our construction.

\section{The case of  measure spaces $(X,{\mathscr B}, \mu)$ with no free ultrafilter in ${\mathscr B}$ having c.i.p.}

In this section we show that  for certain classes of measure spaces $(X,{\mathscr B}, \mu)$, the structure of measure spaces $(Y,{\mathscr C}, \lambda)$ in which  $(X,{\mathscr B}, \mu)$ is embedded is expressible in a simpler way: They are simply obtained by ``blowing" certain points of $X$  up and ``pasting" a certain measurable space to $X$ in a certain  way. This we prove in the next theorem.  Examples of measure spaces satisfying this assumption  are given in the next section.

\begin{theorem}\label{FDFB}
Let $(X,{\mathscr B},\mu)$ be a measure space. Suppose that the points of $X$ are separated by measurable sets in ${\mathscr B}$ and that  there is no free ultrafilter in  ${\mathscr B}$ with c.i.p. Then $(Y,{\mathscr C},\lambda)$ is a measure space in which $(X,{\mathscr B},\mu)$ is embedded if and only if
there exists a measurable space $(Z,{\mathscr D})$, a collection $\{{\mathscr D}_B:B\in{\mathscr B}\}$ of non-empty subsets of ${\mathscr D}$ such that
\begin{itemize}
\item[\rm(1)] $\emptyset\in{\mathscr D}_\emptyset$;
\item[\rm(2)] if $B\in{\mathscr B}$, then
\[\{Z\backslash D:D\in {\mathscr D}_B\}\subseteq{\mathscr D}_{X\backslash B};\]
\item[\rm(3)] if $B_1,B_2,\ldots\in {\mathscr B}$, then
\[\Big\{\bigcup_{n=1}^\infty D_n :D_n\in {\mathscr D}_{B_n}\Big\}\subseteq{\mathscr D}_{\;\bigcup_{n=1}^\infty B_n};\]
\end{itemize}
and a collection $\{T_u:u\in U\}$ of pairwise disjoint non-empty sets, bijectively indexed by a subset $U$ of $X$, where the sets $X$, $T_u$ for any $u\in U$, and $Z$ are pairwise disjoint, such that
\[Y=X\cup \bigcup_{u\in U} T_u\cup Z,\]
\[{\mathscr C}=\Big\{B\cup \bigcup_{u\in B\cap U} T_u\cup D:B\in {\mathscr B}\mbox{ and }D\in {\mathscr D}_B\Big\}\]
and $\lambda:{\mathscr C}\rightarrow [0,\infty]$ is given by
\[\lambda\Big(B\cup \bigcup_{u\in B\cap U} T_u\cup D\Big)=\mu(B)\]
for each $B\in {\mathscr B}$ and $D\in {\mathscr D}_B$.
\end{theorem}

\begin{proof}
Suppose that $(Y,{\mathscr C},\lambda)$ is  a measure space in which $(X,{\mathscr B},\mu)$ is embedded. Assume the representation given for $(Y,{\mathscr C},\lambda)$ in Theorem \ref{FDDDB}. Assume the notation of Theorem \ref{FDDDB}. By our assumption for each ${\mathscr U}\in {\mathbb U}$ the set $\bigcap{\mathscr U}$ is non-empty. Note that $\bigcap{\mathscr U}$ is a singleton, as if $x,z\in\bigcap{\mathscr U}$ and $x\neq z$, then by our assumption $x\in B$ and $z\notin B$ for some $B\in{\mathscr B}$. Now $X\backslash B\in{\mathscr B}$ intersects each element of ${\mathscr U}$, thus $X\backslash B\in{\mathscr U}$. This contradicts the fact that $x\notin X\backslash B$. Let
\[\bigcap{\mathscr U}=\{u_{{\mathscr U}}\}.\]
Define
\[U=\{u_{{\mathscr U}}:{\mathscr U}\in {\mathbb U}\}.\]

\medskip

\noindent {\bf Claim 1.} {\em For each ${\mathscr U}\in {\mathbb U}$ we have
\[{\mathscr U}=\{B\in {\mathscr B}:u_{{\mathscr U}}\in B\}. \]}

\medskip

\noindent {\em Proof of Claim 1.}  Let ${\mathscr U}\in {\mathbb U}$. Obviously, $u_{{\mathscr U}}\in B$ for each $B\in{\mathscr U}$. Conversely, if $B\in{\mathscr B}$ is such that $u_{{\mathscr U}}\in B$ then $B\in {\mathscr U}$. As otherwise, $B\cap G=\emptyset$ for some $G\in {\mathscr U}$. Thus $G\subseteq X\backslash B\in {\mathscr B}$ which implies that $X\backslash B\in {\mathscr U}$. This contradicts the fact that $u_{{\mathscr U}}\notin X\backslash B$ and proves the claim.

\medskip

\noindent For each $u\in U$ there exists some ${\mathscr U}\in {\mathbb U}$ such that $u=u_{{\mathscr U}}$. Note that by Claim 1 such a ${\mathscr U}$ is unique; let $T_u=S_{\mathscr U}$. The collection
\[\{T_u:u\in U\}\]
consists of non-empty sets which are pairwise disjoint and  bijectively indexed (as any distinct $u,v\in U$ are of the form  $u=u_{{\mathscr U}}$ and  $v=u_{{\mathscr V}}$ for some distinct ${\mathscr U}, {\mathscr V}\in  {\mathbb U}$).

\medskip

\noindent {\bf Claim 2.} {\em For each $B\in {\mathscr B}$ we have
\[\bigcup_{u\in B\cap U} T_u=\bigcup_{B\in {\mathscr U}\in {\mathbb U}} S_{{\mathscr U}}. \]}

\medskip

\noindent {\em Proof of Claim 2.} Suppose that $B\in {\mathscr U}\in {\mathbb U}$. By Claim 1 we have $u_{{\mathscr U}}\in B\cap U$. Note that $T_{u_{{\mathscr U}}}=S_{{\mathscr U}}$. To show the reverse inclusion, let $u\in B\cap U$. Let ${\mathscr U}\in {\mathbb U}$ be such that $u=u_{{\mathscr U}}$. Then, by our definition  $T_u=S_{\mathscr U}$. Note that by Claim 1 we have $B\in {\mathscr U}$. This proves the claim.

\medskip

\noindent Note that if $B=X$ then Claim 2 implies that
\[\bigcup_{u\in U} T_u=\bigcup_{{\mathscr U}\in {\mathbb U}} S_{{\mathscr U}}. \]
From the above the desired representation of $(Y,{\mathscr C}, \lambda)$ follows.

Conversely, suppose that $Y$, ${\mathscr C}$ and $\lambda$ are given as in the statement of the theorem. Assume the notation of the theorem. For each $u\in U$ define
\[{\mathscr U}_u=\{B\in {\mathscr B}:u\in B\}.\]

\medskip

\noindent {\bf Claim 3.} {\em For each $u\in U$ the set ${\mathscr U}_u$ is an ultrafilter in  ${\mathscr B}$ with c.i.p.}

\medskip

\noindent {\em Proof of Claim 3.} Let  $u\in U$. Obviously, $\emptyset\neq {\mathscr U}_u\subseteq{\mathscr B}$, $\emptyset\notin{\mathscr U}_u$ and  ${\mathscr U}_u$ is closed under finite intersections. Also, note that if $G\subseteq B$ for some $G\in{\mathscr U}_u$ and $B\in{\mathscr B}$, then $B\in{\mathscr U}_u$. Thus  ${\mathscr U}_u$ is a filter in ${\mathscr B}$. To show that ${\mathscr U}_u$ is an ultrafilter, suppose that $B\in {\mathscr B}$ is such that $B\cap G$ is non-empty for each $G\in{\mathscr U}_u$. If $B\notin {\mathscr U}_u$ then $u\notin B$. Thus $u\in X\backslash B\in {\mathscr B}$ and $X\backslash B\in{\mathscr U}_u$. But  $X\backslash B$ misses $B$, which is a contradiction. Therefore $B\in {\mathscr U}_u$. The fact that ${\mathscr U}_u$ has c.i.p. is obvious.  This proves the claim.

\medskip

\noindent Let
\[{\mathbb U}=\{{\mathscr U}_u:u\in U\}.\]
Note that each ${\mathscr U}\in {\mathbb U}$ is of the form ${\mathscr U}_u$ for some unique $u\in U$. This is because if $ {\mathscr U}_u={\mathscr U}={\mathscr U}_v$ for some distinct $u,v\in U$, then by our assumption $u\in B$ and $v\notin B$ for some $B\in{\mathscr B}$. Thus $B\in{\mathscr U}_u\backslash{\mathscr U}_v$. Therefore ${\mathscr U}_u\neq{\mathscr U}_v$, which is a contradiction. For each  ${\mathscr U}\in{\mathbb U}$ define $S_{{\mathscr U}}=T_u$, where $u\in U$ is such that ${\mathscr U}={\mathscr U}_u$. The collection
\[\{S_{{\mathscr U}}:{\mathscr U}\in{\mathbb U}\}\]
consists of non-empty sets which are pairwise disjoint and  bijectively indexed (as distinct elements of ${\mathbb U}$ are assigned to distinct elements of $U$).

\medskip

\noindent {\bf Claim 4.} {\em For each $B\in {\mathscr B}$ we have
\[\bigcup_{B\in {\mathscr U}\in {\mathbb U}} S_{{\mathscr U}}=\bigcup_{u\in B\cap U} T_u. \]}

\medskip

\noindent {\em Proof of Claim 4.} Let $u\in B\cap U$. Then, by our definition of  ${\mathscr U}_u$ we have $B\in {\mathscr U}_u\in{\mathbb U}$. Also, by our definition $T_u=S_{{\mathscr U}_u}$. To show the reverse inclusion, let $B\in {\mathscr U}\in {\mathbb U}$. Let $u\in U$ be such that ${\mathscr U}={\mathscr U}_u$. Then, by our definition $S_{\mathscr U}=T_u$. But since $B\in {\mathscr U}_u$, by our definition we have $u\in B$, and thus $u\in B\cap U$. This proves the claim.

\medskip

\noindent  Note that if $B=X$ then Claim 4 implies that
\[\bigcup_{{\mathscr U}\in {\mathbb U}} S_{{\mathscr U}}=\bigcup_{u\in U} T_u. \]
From the above  and Theorem \ref{FDDDB} the result  follows.
\end{proof}

\section{Examples of  measure spaces $(X,{\mathscr B},\mu)$ with no free ultrafilter in ${\mathscr B}$ having c.i.p.}

In this section we give examples of  measure spaces $(X,{\mathscr B},\mu)$ for which the assumption of Theorem \ref{FDFB} holds, i.e.,  measure spaces $(X,{\mathscr B}, \mu)$ for which there is  no free ultrafilter in ${\mathscr B}$ with c.i.p.  The following gives some  equivalent ways to express this condition. The equivalence of conditions (1) and (2) in the following proposition is well known; we include the proof in here for the sake of completeness.

\begin{proposition}\label{JJHHC}
Let  $(X, {\mathscr B})$ be a  measurable  space. Then the following are equivalent:
\begin{itemize}
\item[\rm(1)] There is no free ultrafilter in  ${\mathscr B}$  with c.i.p.
\item[\rm(2)] For every $\{0,1\}$-valued measure $\mu$ on ${\mathscr B}$ whose null sets cover $X$ we have $\mu\equiv 0$.
\item[\rm(3)] For every $\{0,1\}$-valued measure $\mu$ on ${\mathscr B}$ whose null sets cover $X$, if  ${\mathscr C}\subseteq{\mathscr B}$ is non-empty and  such that $\bigcup{\mathscr C}\in {\mathscr B}$, then
    \[\mu\Big(\bigcup{\mathscr C}\Big)=\sup_{C\in{\mathscr C}}\mu (C).\]
\end{itemize}
\end{proposition}

\begin{proof}
That (2) implies (3) is trivial. (1) {\em implies} (2). Let $\mu$ be a non-trivial  $\{0,1\}$-valued measure on ${\mathscr B}$ whose null sets cover $X$. Define
\[{\mathscr F}=\big\{B\in {\mathscr B}:\mu(B)=1\big\}.\]
We show that ${\mathscr F}$ is a free ultrafilter in  ${\mathscr B}$  with c.i.p. If $F,G\in {\mathscr F}$, then since
\[\mu(F\cup G)=\mu(F\backslash G)+\mu(G)\]
and $\mu(G)=1$, we have $\mu(F\backslash G)=0$. Therefore
\[\mu(F\cap G)=\mu(F\backslash  G)+\mu(F\cap G)=\mu(F)=1\]
and thus $F\cap G\in {\mathscr F}$. Obviously, if $F\subseteq B$ for some $F\in{\mathscr F}$ and $B\in{\mathscr B}$, then $B\in{\mathscr F}$. Therefore, ${\mathscr F}$ is a filter in ${\mathscr B}$. To show that ${\mathscr F}$ is an ultrafilter, let $B\in{\mathscr B}$ be such that $B\cap F$ is non-empty for each $F\in{\mathscr F}$. If $B\notin{\mathscr F}$, then  $\mu(B)=0$, and thus, since $\mu(X\backslash B)=1$, we have $X\backslash B\in{\mathscr F}$. But this is not possible, as $B$ misses  $X\backslash B$. Therefore  $B\in{\mathscr F}$. To show that ${\mathscr F}$ has c.i.p., let $F_1,F_2,\ldots \in{\mathscr F}$. Without any loss of generality we may assume that $F_1\supseteq F_2\supseteq\cdots$. If
\[\bigcap_{n=1}^\infty F_n\notin{\mathscr F}\]
then
\[\mu\Big(\bigcap_{n=1}^\infty F_n\Big)=0.\]
Thus (with the empty intersection interpreted as $X$) we have
\begin{equation}\label{JJUY}
1=\mu\Big(X\backslash\bigcap_{n=1}^\infty F_n\Big)=\mu\Big(\bigcup_{n=1}^\infty (F_{n-1}\backslash F_n)\Big)=\sum_{n=1}^\infty\mu(F_{n-1}\backslash F_n).
\end{equation}
Now, each $G_n=F_{n-1}\backslash F_n$, where $n\in\mathbb{N}$, misses $F_n\in {\mathscr F}$ and thus $G_n\notin {\mathscr F}$; therefore $\mu(G_n)=0$. This contradicts (\ref{JJUY}) and shows that
\[\bigcap_{n=1}^\infty F_n\in {\mathscr F}.\]
To show that ${\mathscr F}$ is free, let $x\in X$. By our assumption $x\in B$ for some $B\in {\mathscr B}$ such that $\mu(B)=0$. Thus $\mu(X\backslash B)=1$. But $X\backslash B\in{\mathscr F}$ and $x\notin X\backslash B$. Therefore $\bigcap{\mathscr F}=\emptyset$.

(3) {\em implies} (1). Suppose  that  there exists a  free ultrafilter  ${\mathscr F}$ in  ${\mathscr B}$  with c.i.p. Define  $\mu:{\mathscr B}\rightarrow \{0,1\}$ such that  $\mu(B)=1$ if $B\in {\mathscr F}$ and $\mu(B)=0$ if $B\notin {\mathscr F}$. To show that $\mu$ is a measure, first note that $\mu(\emptyset)=0$. Let $B_1,B_2,\ldots \in{\mathscr B}$ be pairwise disjoint. Suppose that
\[\bigcup_{n=1}^\infty B_n\notin {\mathscr F}.\]
Then $ B_n\notin {\mathscr F}$ for some $n\in\mathbb{N}$. Therefore
\begin{equation}\label{HUY}
\mu\Big(\bigcup_{n=1}^\infty B_n\Big)=\sum_{n=1}^\infty\mu(B_n)
\end{equation}
as each side is identical to $0$. Suppose that
\[\bigcup_{n=1}^\infty B_n\in {\mathscr F}.\]
By Lemma \ref{LJHG} this implies that $ B_n\in {\mathscr F}$ for some $n\in\mathbb{N}$. Note that for each $n\neq i\in\mathbb{N}$, since $B_i\cap B_n=\emptyset$ we have $B_i\notin  {\mathscr F}$. Thus (\ref{HUY}) holds, as in this case each side is identical to $1$. This shows that $\mu$ is a measure. Since   ${\mathscr F}$ is free, for each $x\in X$ there exists some $F\in {\mathscr F}$ such that $x\notin F$. Thus $x\in X\backslash F$, and since $X\backslash F\notin {\mathscr F}$ we have $\mu(X\backslash F)=0$. Therefore the null sets of ${\mathscr B}$ cover $X$. Now, let
\[{\mathscr C}=\{X\backslash F: F\in {\mathscr F}\}.\]
Then since ${\mathscr F}$ is free, we have $\bigcup{\mathscr C}=X$, and therefore by our assumption
\[\sup_{F\in{\mathscr F}}\mu(X\backslash F)=\sup_{C\in{\mathscr C}}\mu (C)=\mu\Big(\bigcup{\mathscr C}\Big)=\mu (X)=1.\]
But this is not possible, as if $F\in{\mathscr F}$ then $X\backslash F\notin{\mathscr F}$, as it misses $F$, and thus $\mu(X\backslash F)=0$.
\end{proof}

\begin{theorem}\label{FDOH}
Let $(X,{\mathscr B})$ be a measurable  space. If  $X$ is countable then there is no free ultrafilter in ${\mathscr B}$ with c.i.p.
\end{theorem}

\begin{proof}
Let $X=\{x_1,x_2,\ldots\}$ and let ${\mathscr F}$ be a free  ultrafilter in ${\mathscr B}$. Since ${\mathscr F}$ is free for each $n\in\mathbb{N}$ there is some $F_n\in{\mathscr F}$ such that $x_n\notin F_n$. Then
\[\bigcap_{n=1}^\infty F_n=\emptyset\]
and thus ${\mathscr F}$ does not have  c.i.p.
\end{proof}

\begin{theorem}\label{FDDHH}
Let $(Y,{\mathscr C})$ be a measurable space. Let $X\in{\mathscr C}$ and
\[{\mathscr B}=\{C\in{\mathscr C}: C\subseteq X \},\]
i.e., $X\in{\mathscr C}$ and $(X,{\mathscr B})$ is embedded in $(Y,{\mathscr C})$. If there is no free ultrafilter in ${\mathscr C}$ with c.i.p. then there is no free ultrafilter in ${\mathscr B}$  with c.i.p.
\end{theorem}

\begin{proof}
Simply note that if ${\mathscr F}$ is a free ultrafilter  in ${\mathscr B}$  with c.i.p., then
\[{\mathscr G}=\{G\in {\mathscr C}:G\supseteq F\mbox{ for some }F\in{\mathscr F} \}\]
is a filter in ${\mathscr C}$ which is  an ultrafilter (as if $C\in {\mathscr C}$ meets each $G\in {\mathscr G}$ then, since ${\mathscr F}\subseteq{\mathscr G}$, the set $(X\cap C)\cap F=C\cap F$ is non-empty for each $F\in{\mathscr F}$ and thus $X\cap C\in{\mathscr F}$ which implies that $C\in {\mathscr G}$, as $X\cap C\subseteq C$) and it is free (as ${\mathscr F}\subseteq{\mathscr G}$ and thus $\bigcap{\mathscr G}\subseteq\bigcap{\mathscr F}$ and ${\mathscr F}$ is free) and has c.i.p. (as ${\mathscr F}$ has, and if $G_1,G_2,\ldots\in{\mathscr G}$ then
\[\bigcap_{n=1}^\infty G_n\supseteq\bigcap_{n=1}^\infty F_n,\]
where for each $n\in\mathbb{N}$, the element $F_n\in{\mathscr F}$ is such that $F_n\subseteq G_n$).
\end{proof}

\begin{theorem}\label{JLC}
Let $(X,{\mathscr B})\subseteq(Y,{\mathscr C})$ and let $Y\backslash X$ be countable. If there is no free ultrafilter in ${\mathscr B}$ with c.i.p. then there is no free ultrafilter in ${\mathscr C}$ with c.i.p.
\end{theorem}

\begin{proof}
Let
\[Y\backslash X=\{y_1,y_2,\ldots\}.\]
Let ${\mathscr H}$ be a free ultrafilter in ${\mathscr C}$ with c.i.p. Since ${\mathscr H}$ is free, for any $n\in\mathbb{N}$ there exists some $H_n\in {\mathscr H}$ such that $y_n\notin H_n$. Let
\[{\mathscr A}=\{H\cap X:H\in{\mathscr H}\}\subseteq{\mathscr B}.\]
Note that $\emptyset\notin{\mathscr A}$, as otherwise $H\cap X=\emptyset$ for some $H\in{\mathscr H}$. Now
\[H\cap\bigcap_{n=1}^\infty H_n=\emptyset,\]
as
\[H\cap\bigcap_{n=1}^\infty H_n\subseteq (X\cap H)\cup\Big((Y\backslash X)\cap\bigcap_{n=1}^\infty H_n\Big)=\emptyset\]
contradicting the fact that ${\mathscr H}$ has c.i.p. Thus ${\mathscr A}$ is a filter-base in ${\mathscr B}$, as ${\mathscr A}$ is closed under finite intersections. Let ${\mathscr F}$ be an ultrafilter in ${\mathscr B}$ such that ${\mathscr A}\subseteq{\mathscr F}$. Then since
\[\bigcap{\mathscr F}\subseteq\bigcap{\mathscr A}=\bigcap{\mathscr H}\cap X\]
(and ${\mathscr H}$ is free), ${\mathscr F}$  is free. To show that ${\mathscr F}$ has c.i.p., let $F_1,F_2,\ldots\in{\mathscr F}$. Let $F_n=C_n\cap X$ where $C_n\in{\mathscr C}$ for each $n\in\mathbb{N}$. Let $n\in\mathbb{N}$. For each $H\in {\mathscr H}$, since $H\cap X\in {\mathscr A}\subseteq{\mathscr F}$ we have
\[C_n\cap H\cap X=F_n\cap H\cap X\in {\mathscr F}.\]
Therefore $H\cap C_n$ is non-empty and thus (since ${\mathscr H}$ is an ultrafilter) $C_n\in {\mathscr H}$. Now, since
\begin{eqnarray*}
\bigcap_{n=1}^\infty C_n\cap\bigcap_{n=1}^\infty H_n&\subseteq&\Big(X\cap\bigcap_{n=1}^\infty C_n\Big)\cup\Big((Y\backslash X)\cap\bigcap_{n=1}^\infty H_n\Big)\\&=&\bigcap_{n=1}^\infty (X\cap C_n)=\bigcap_{n=1}^\infty F_n
\end{eqnarray*}
and ${\mathscr H}$ has c.i.p., we have
\[\bigcap_{n=1}^\infty F_n\neq\emptyset.\]
Therefore ${\mathscr F}$ has c.i.p.
\end{proof}

If $(X,{\mathscr B})$ and $(Y,{\mathscr C})$ are measurable spaces, we denote by ${\mathscr B}\times{\mathscr C}$ the smallest $\sigma$-algebra on $X\times Y$ containing the set
\[\{B\times C:B\in{\mathscr B}\mbox{ and }C\in{\mathscr C}\}\]
of all measurable rectangles in $X\times Y$.

\begin{theorem}\label{JLFC}
Let $(X,{\mathscr B})$ and $(Y,{\mathscr C})$ be  measurable spaces such that in each of them singletons are measurable. Then the following are equivalent:
\begin{itemize}
\item[\rm(1)] There is no free ultrafilter  in ${\mathscr B}\times{\mathscr C}$  with c.i.p.
\item[\rm(2)] Neither there is a free ultrafilter  in ${\mathscr B}$ with c.i.p. nor  there is a free ultrafilter  in ${\mathscr C}$ with c.i.p.
\end{itemize}
\end{theorem}

\begin{proof}
(1) {\em implies} (2). Let  ${\mathscr F}$ be a free ultrafilter  in ${\mathscr B}$ with c.i.p. Fix some $y\in Y$ and let
\[{\mathscr A}=\big\{F\times\{y\}:F\in {\mathscr F}\big\}.\]
Then ${\mathscr A}$ is a filter-base in  ${\mathscr B}\times{\mathscr C}$, as $\emptyset\notin{\mathscr A}$ and ${\mathscr A}$ is closed under finite intersections. Let ${\mathscr H}$ be an ultrafilter in ${\mathscr B}\times{\mathscr C}$  such that ${\mathscr A}\subseteq{\mathscr H}$. Then
\[\bigcap{\mathscr H}\subseteq\bigcap{\mathscr A}=\Big(\bigcap{\mathscr F}\Big)\times\{y\}=\emptyset\]
and ${\mathscr H}$ is also free. We verify that ${\mathscr H}$ has c.i.p. Let $H_1,H_2,\ldots\in{\mathscr H}$. Let
\[H_n^y=\big\{x\in X:(x,y)\in H_n\big\}\]
for each $n\in\mathbb{N}$. Then $H_n^y$, for each $n\in\mathbb{N}$, being the $y$-section of the measurable set $H_n$ of ${\mathscr B}\times{\mathscr C}$ is measurable in $X$. Note that for each $n\in\mathbb{N}$ and $F\in{\mathscr F}$, since $F\times\{y\}\in {\mathscr H}$, we have
\[H_n\cap\big(F\times\{y\}\big)\neq\emptyset,\]
and thus $H_n^y\cap F$ is non-empty. Therefore, since ${\mathscr F}$ is an ultrafilter, we have $H_n^y\in {\mathscr F}$ for each $n\in\mathbb{N}$. Since ${\mathscr F}$ has c.i.p., we have
\[\bigcap_{n=1}^\infty H_n^y\neq\emptyset.\]
Now since
\[\Big(\bigcap_{n=1}^\infty H_n^y\Big)\times\{y\}\subseteq\bigcap_{n=1}^\infty H_n\]
it follows that
\[\bigcap_{n=1}^\infty H_n\neq\emptyset.\]
A similar argument can be used in the case when there is a free ultrafilter in ${\mathscr C}$ with c.i.p.

(2) {\em implies} (1). Let ${\mathscr H}$  be a  free ultrafilter in  ${\mathscr B}\times{\mathscr C}$ with c.i.p. Let
\[{\mathscr A}=\{B\in {\mathscr B}:B\times C\in{\mathscr H}\mbox{ for some } C\in {\mathscr C} \}.\]
Then ${\mathscr A}$ is a filter-base in ${\mathscr B}$, as $\emptyset\notin{\mathscr A}$ and it is closed under finite intersections. Let ${\mathscr F}$ be an ultrafilter in ${\mathscr B}$ such that ${\mathscr A}\subseteq{\mathscr F}$.

\medskip

\noindent {\bf Claim.} {\em For each $F\in{\mathscr F}$ we have $F\times Y\in{\mathscr H}$.}

\medskip

\noindent {\em Proof of the claim.}  Otherwise, if $F\times Y\notin {\mathscr H}$ for some $F\in {\mathscr F}$, then since  ${\mathscr H}$ is an ultrafilter, we have $(F\times Y)\cap H=\emptyset$ for some $H\in {\mathscr H}$. Thus $H\subseteq (X\backslash F)\times Y$ and therefore $(X\backslash F)\times Y\in {\mathscr H}$.  But this implies that $X\backslash F\in {\mathscr A}$, and thus $X\backslash F\in {\mathscr F}$, which is not possible, as it misses $F\in {\mathscr F}$.

\medskip

\noindent Now, we show that ${\mathscr F}$ has c.i.p. Let $F_1,F_2,\ldots \in {\mathscr F}$.  By the above $F_n\times Y\in {\mathscr H}$ for each $n\in\mathbb{N}$ and thus, since ${\mathscr H}$ has c.i.p.,
\[\bigcap_{n=1}^\infty (F_n\times Y)\neq\emptyset.\]
But
\[\bigcap_{n=1}^\infty (F_n\times Y)=\Big(\bigcap_{n=1}^\infty F_n\Big)\times Y.\]
Therefore
\[\bigcap_{n=1}^\infty F_n\neq\emptyset.\]
By our assumption ${\mathscr F}$ is not free, i.e., $\bigcap{\mathscr F}$ is non-empty. Let $p\in \bigcap{\mathscr F}$. Then $\{p\}\in{\mathscr B}$ meets each $F\in {\mathscr F}$, and thus, since  ${\mathscr F}$ is an ultrafilter we have $\{p\}\in{\mathscr F}$. By the claim $\{p\}\times Y\in {\mathscr H}$. Similarly, there exists some $q\in Y$ such that $X\times\{q\}\in{\mathscr H}$. Now
\[\big\{(p,q)\big\}=\big(\{p\} \times Y\big)\cap\big(X\times\{q\}\big)\in {\mathscr H}\]
and thus, since $\{(p,q)\}\cap H$ is non-empty for each $H\in {\mathscr H}$, we have $(p,q)\in \bigcap{\mathscr H}$, which is a contradiction.
\end{proof}

Let $X$ be a completely regular topological space. Recall that {\em the $\sigma$-algebra of Baire subsets of $X$} (denoted by  ${\mathscr B}^*(X)$) is the smallest  $\sigma$-algebra in $X$ containing ${\mathscr Z}(X)$. By a {\em Baire measure on $X$} we mean a finite measure on  ${\mathscr B}^*(X)$. The {\em support} of a Baire measure $\mu$ on $X$ is defined to be the set
\[\big\{x\in X:\mu(U)>0\mbox{ for every }U\in\mathrm{Coz}(X)\mbox{ such that }x\in U\big\}\]
and is denoted by $\mathrm{supp}(\mu)$.

The following Lemma is well known. (See \cite{GP}.)

\begin{lemma}\label{FRAYU}
Let $X$ be a completely regular topological  space.
Then the following are equivalent:
\begin{itemize}
\item[\rm(1)] $X$ is realcompact.
\item[\rm(2)] Each $\{0,1\}$-valued non-trivial Baire measure on $X$ has a non-empty support.
\item[\rm(3)] Each ultrafilter in ${\mathscr Z}(X)$ with c.i.p. is fixed.
\end{itemize}
\end{lemma}

\begin{theorem}\label{JEDC}
Let  $(X,{\mathscr O}, {\mathscr B})$ be a  first-countable realcompact topological  measurable space. Then there is no free ultrafilter in ${\mathscr B}$ with c.i.p.
\end{theorem}

\begin{proof}
We prove the theorem in two different ways. Our first approach, which is rather direct, is more topological; our second  approach makes use of the characterization given in Lemma \ref{FRAYU}.

\subsubsection*{First approach.} Suppose to the contrary that there exists a free  ultrafilter  ${\mathscr F}$ in ${\mathscr B}$  with c.i.p. Note that the collection ${\mathscr B}$ of all measurable sets can be considered as a base for a topology on $X$, as it is closed under finite intersections and covers $X$. Denote by ${\mathscr O}_{\mathscr B}$ the topology generated by ${\mathscr B}$ on $X$.  Since $(X,{\mathscr O})$ is Hausdorff and ${\mathscr O}\subseteq {\mathscr B}$, the topological space $(X,{\mathscr O}_{\mathscr B})$ is Hausdorff and therefore completely regular, as the elements of ${\mathscr B}$ are closed-open in $(X,{\mathscr O}_{\mathscr B})$. Let
\[\phi:\beta(X,{\mathscr O}_{\mathscr B})\rightarrow\beta(X,{\mathscr O})\]
continuously extend
\[\mathrm{id}_X:(X,{\mathscr O}_{\mathscr B})\rightarrow(X,{\mathscr O}).\]
Since
\[\{\mathrm{cl}_{\beta(X,{\mathscr O}_{\mathscr B})}F:F\in {\mathscr F}\}\]
has the finite  intersection property, as ${\mathscr F}$ has, by compactness we have
\[G=\bigcap\{\mathrm{cl}_{\beta(X,{\mathscr O}_{\mathscr B})}F:F\in {\mathscr F}\}\neq\emptyset.\]
Let $p\in G$. Note that since ${\mathscr F}$ is free we have $p\notin X$, as otherwise
\[p\in X\cap\mathrm{cl}_{\beta(X,{\mathscr O}_{\mathscr B})}F=F\]
for each $F\in{\mathscr F}$.

\medskip

\noindent {\bf Claim.} {\em If $V$ is an open neighborhood of $\phi(p)$ in $\beta(X,{\mathscr O})$ then $V\cap X\in {\mathscr F}$.}

\medskip

\noindent {\em Proof of the claim.} Suppose the contrary, i.e., suppose that $V\cap X\notin {\mathscr F}$. Note that $V\cap X\in{\mathscr O}\subseteq{\mathscr B}$. Since ${\mathscr F}$  is an ultrafilter, $V\cap X\cap F=\emptyset$ for some $F\in{\mathscr F}$. Since $V\cap X\in{\mathscr B}$ and $F\in {\mathscr B}$, the sets $V\cap X$ and $F$  are closed-open in $(X,{\mathscr O}_{\mathscr B})$, and therefore
\[\mathrm{cl}_{\beta(X,{\mathscr O}_{\mathscr B})}(V\cap X)\cap\mathrm{cl}_{\beta(X,{\mathscr O}_{\mathscr B})}F=\mathrm{cl}_{\beta(X,{\mathscr O}_{\mathscr B})}(V\cap X\cap F)=\emptyset.\]
By the choice of $p$ we have
\[p\notin\mathrm{cl}_{\beta(X,{\mathscr O}_{\mathscr B})}(V\cap X).\]
Let $W$ be an open neighborhood of $p$ in $\beta(X,{\mathscr O}_{\mathscr B})$ such that $W\cap V\cap X=\emptyset$. By continuity of $\phi$ there exists an open neighborhood $U$ of $p$ in $\beta(X,{\mathscr O}_{\mathscr B})$ such that $\phi(U)\subseteq V$. Then (since $\phi|X=\mathrm{id}_X$) we have
\[U\cap W\cap X=\phi(U\cap W\cap X)\subseteq\phi(U)\cap W\cap X\subseteq V\cap W\cap X=\emptyset.\]
But this is a contradiction, as $U\cap W$, being a non-empty open subset of  $\beta(X,{\mathscr O}_{\mathscr B})$, meets $X$. Thus $V\cap X\in {\mathscr F}$.

\medskip

\noindent Next, we show that $\phi(p)\notin X$. Suppose the contrary. By our assumption, there exists a countable base
\[\{V_n:n\in\mathbb{N}\}\]
at $\phi(p)$ in  $(X,{\mathscr O})$. By the claim, $V_n\in {\mathscr F}$ for each $n\in\mathbb{N}$. Now, since ${\mathscr F}$ has c.i.p., for each $F\in {\mathscr F}$ we have
\[F\cap\big\{\phi(p)\big\}=F\cap\bigcap_{n=1}^\infty V_n\neq\emptyset.\]
Thus $\phi(p)\in\bigcap{\mathscr F}$, which is a contradiction, as ${\mathscr F}$ is free. Therefore  $\phi(p)\in \beta(X,{\mathscr O})\backslash X$. Since, by our assumption $(X,{\mathscr O})$ is realcompact, there exists a zero-set $Z$ in $\beta(X,{\mathscr O})$ such that $\phi(p)\in Z$ and $Z\cap X=\emptyset$. Let $Z=f^{-1}(0)$ for some continuous $f: \beta(X,{\mathscr O})\rightarrow [0,1]$. Now, for each  $n\in\mathbb{N}$ the set $f^{-1}([0,1/n))$ is an open neighborhood of $\phi(p)$ in $\beta(X,{\mathscr O})$, thus by the claim
\[f^{-1}\big([0,1/n)\big)\cap X\in {\mathscr F}.\]
Since ${\mathscr F}$ has c.i.p.  we have
\[Z \cap X=f^{-1}(0)\cap X=\bigcap_{n=1}^\infty f^{-1}\big([0,1/n)\big)\cap X\neq\emptyset\]
which contradicts the choice of $Z$.

\subsubsection*{Second approach.} Suppose to the contrary that there exists a free  ultrafilter  ${\mathscr F}$ in ${\mathscr B}$  with c.i.p.  Define $\nu:{\mathscr B}\rightarrow\{0,1\}$ such that $\nu(B)=0$ if $B\notin {\mathscr F}$ and $\nu(B)=1$ if $B\in {\mathscr F}$. Then (as in the proof of Proposition \ref{JJHHC} (3)$\Rightarrow$(1)) $\nu$ is a measure. Note that ${\mathscr B}$ contains the set ${\mathscr B}^*(X)$ of Baire subsets of $X$ (as each zero-set in $X$, being a $G_\delta$, is contained in ${\mathscr B}$). Denote
\[\mu=\nu|{\mathscr B}^*(X):{\mathscr B}^*(X)\rightarrow\{0,1\}.\]
Then $\mu$ is a non-trivial Baire measure on $X$, and since we are assuming that $X$ is realcompact, by Lemma \ref{FRAYU} it has non-empty support. Let $x\in \mathrm{supp}(\mu)$ and let $\{C_n:n\in\mathbb{N}\}$ be a local base at $x$ in $X$. Without any loss of generality (since $X$ is completely regular) we may assume that $C_n\in\mathrm{Coz}(X)$ for each $n\in\mathbb{N}$ and $C_1\supseteq C_2\supseteq\cdots$. Then
\[\mu(C_n)\rightarrow\mu\Big(\bigcap_{n=1}^\infty C_n\Big).\]
But this is a contradiction, as (since $x\in\mathrm{supp}(\mu)$) $\mu(C_n)=1$ for each $n\in\mathbb{N}$, and since  ${\mathscr F}$ is free,
\[\bigcap_{n=1}^\infty C_n=\{x\}\notin{\mathscr F};\]
as otherwise, $x\in F$ for each $F\in{\mathscr F}$, as $\{x\}\cap F$ is non-empty, and thus (by our definition of $\nu$)
\[\mu\Big(\bigcap_{n=1}^\infty C_n\Big)=0.\]
\end{proof}

Obviously, every $n$-dimensional Euclidean space $\mathbb{R}^n$, where $n\in\mathbb{N}$, is realcompact. Thus, from Theorem \ref{JEDC} we obtain the following.

\begin{corollary}\label{JHEDC}
Let  $(\mathbb{R}^n, {\mathscr M})$ be the Lebesgue measurable space, where $n\in\mathbb{N}$. Then there is no free ultrafilter in ${\mathscr M}$ with c.i.p.
\end{corollary}

Recall that a cardinal  $\zeta$ is said to be {\em measurable} if there is a non-trivial  $\{0,1\}$-valued measure defined on the power set ${\mathscr P}(X)$ of a set $X$ of cardinality $\zeta$ which vanishes at singletons. The following is well known.

\begin{theorem}\label{KEDC}
In the measurable space $(X,{\mathscr B})$, let ${\mathscr B}={\mathscr P}(X)$. Then the following are equivalent:
\begin{itemize}
\item[\rm(1)] There is no free ultrafilter in ${\mathscr B}$ with c.i.p.
\item[\rm(2)] $X$ is of non-measurable cardinality.
\end{itemize}
\end{theorem}

Our next theorem  is sort of converse to Theorem \ref{JEDC}. We need, however, some definitions and some lemmas first.

Let $X$ be a completely regular topological space. For an open subset $U$ of $X$ the {\em extension of $U$ to $\beta X$} is defined by
\[\mathrm{Ex}_X U=\beta X\backslash\mathrm{cl}_{\beta X}(X\backslash U).\]

The following lemma is well known (see Lemma
7.1.13 of \cite{E} or  Lemma 3.1 of \cite{vD}).

\begin{lemma}\label{TYU}
Let $X$ be a completely regular topological space and let $U$ and $V$ be open subsets of $X$. Then
\begin{itemize}
\item[\rm(1)] $X\cap\mathrm{Ex}_XU=U$, and thus $\mathrm{cl}_{\beta X}\mathrm{Ex}_XU=\mathrm{cl}_{\beta X}U$.
\item[\rm(2)] $\mathrm{Ex}_X(U\cap V)=\mathrm{Ex}_XU\cap\mathrm{Ex}_XV$.
\end{itemize}
\end{lemma}

The following lemma is proved by E. G. Skljarenko in \cite{S}. It is
rediscovered by E. K. van Douwen in \cite{vD}.

\begin{lemma}\label{HJH} Let $X$ be a completely regular topological space and let $U$ be an open subset of $X$. Then
\[\mathrm{bd}_{\beta X}\mathrm{Ex}_XU=\mathrm{cl}_{\beta X}\mathrm{bd}_XU.\]
\end{lemma}

\begin{lemma}\label{JEUDC}
Let  $X$ be a completely regular topological space. If $B$ is a subset of $X$ with compact boundary then
\[\mathrm{cl}_{\beta X}B\backslash X=\mathrm{Ex}_X(\mathrm{int}_XB)\backslash X.\]
\end{lemma}

\begin{proof}
Since $\mathrm{bd}_X B$ is compact, we have
\[\mathrm{cl}_{\beta X}\mathrm{bd}_X B\subseteq\mathrm{bd}_X B\subseteq X,\]
and therefore
\begin{eqnarray*}
\mathrm{cl}_{\beta X}B\backslash X&=&\mathrm{cl}_{\beta X}(\mathrm{int}_XB\cup\mathrm{bd}_XB)\backslash X\\&=&(\mathrm{cl}_{\beta X}\mathrm{int}_XB\cup\mathrm{cl}_{\beta X}\mathrm{bd}_XB)\backslash X\\&=&(\mathrm{cl}_{\beta X}\mathrm{int}_XB\backslash X)\cup (\mathrm{cl}_{\beta X}\mathrm{bd}_XB \backslash X)
=\mathrm{cl}_{\beta X}\mathrm{int}_XB\backslash X.
\end{eqnarray*}
By Lemma \ref{TYU} we have
\begin{eqnarray*}
\mathrm{cl}_{\beta X}\mathrm{int}_XB\backslash X&=&\mathrm{cl}_{\beta X}\mathrm{Ex}_X(\mathrm{int}_XB)\backslash X\\&=&\big(\mathrm{Ex}_X(\mathrm{int}_XB)\cup\mathrm{bd}_{\beta X}\mathrm{Ex}_X(\mathrm{int}_XB)\big)\backslash X\\&=&\big(\mathrm{Ex}_X(\mathrm{int}_XB)\backslash X\big)\cup\big(\mathrm{bd}_{\beta X}\mathrm{Ex}_X(\mathrm{int}_XB)\backslash X\big).
\end{eqnarray*}
But by Lemma \ref{HJH}, we have
\[\mathrm{bd}_{\beta X}\mathrm{Ex}_X(\mathrm{int}_XB)=\mathrm{cl}_{\beta X}\mathrm{bd}_X(\mathrm{int}_XB).\]
On the other hand $\mathrm{bd}_X (\mathrm{int}_XB)\subseteq\mathrm{bd}_X B$. Thus
\[\mathrm{cl}_{\beta X}\mathrm{bd}_X (\mathrm{int}_XB)\subseteq\mathrm{bd}_X B\subseteq X.\]
Combining these, we obtain the result.
\end{proof}

Note that if $U$ is an open subset of a completely regular topological space $X$, then (since $X$ is dense in $\beta X$) we have
\[\mathrm{cl}_{\beta X}U=\mathrm{cl}_{\beta X}(U\cap X).\]
We use this simple observation in the following.

\begin{theorem}\label{JUDC}
Let $(X,{\mathscr O}, {\mathscr B})$ be a completely regular topological  measurable  space. Suppose that each $B\in {\mathscr B}$ has compact boundary in $X$. If there is no free ultrafilter in ${\mathscr B}$ with c.i.p. then $X$ is realcompact.
\end{theorem}

\begin{proof}
Suppose the contrary, i.e., suppose that $X$ is not realcompact. Then $X\neq\upsilon X$. Let $p\in \upsilon X\backslash X$. Let
\[{\mathscr F}=\{B\in{\mathscr B}:p\in\mathrm{cl}_{\beta X}B \}.\]
We show that ${\mathscr F}$ is  a free ultrafilter in ${\mathscr B}$ with c.i.p., contradicting  our assumption.  Note that $\emptyset\neq{\mathscr F}\subseteq{\mathscr B}$ and $\emptyset\notin {\mathscr F}$. Suppose that $F,G\in {\mathscr F}$. Then, by Lemmas \ref{TYU} and \ref{JEUDC} we have
\begin{eqnarray*}
p&\in&(\mathrm{cl}_{\beta X}F\cap\mathrm{cl}_{\beta X} G)\backslash X\\&=&(\mathrm{cl}_{\beta X}F \backslash X)\cap(\mathrm{cl}_{\beta X} G \backslash X)\\&=&\big(\mathrm{Ex}_X(\mathrm{int}_XF)\backslash X\big)\cap\big(\mathrm{Ex}_X(\mathrm{int}_XG)\backslash X\big)\\&=& \big(\mathrm{Ex}_X(\mathrm{int}_XF)\cap\mathrm{Ex}_X(\mathrm{int}_XG)\big)\backslash X\\&=&\mathrm{Ex}_X(\mathrm{int}_XF\cap\mathrm{int}_XG)\backslash X\\&=& \mathrm{Ex}_X\big(\mathrm{int}_X(F\cap G)\big)\backslash X=\mathrm{cl}_{\beta X}(F\cap G)\backslash X.
\end{eqnarray*}
Therefore $p\in\mathrm{cl}_{\beta X}(F\cap G)\backslash X$ and thus $ F\cap G\in {\mathscr F}$. Next, suppose that $F\subseteq B$ for some $B\in {\mathscr B}$ and $F\in {\mathscr F}$. Then
\[p\in\mathrm{cl}_{\beta X}F\subseteq\mathrm{cl}_{\beta X}B\]
and thus $B\in {\mathscr F}$. This shows that ${\mathscr F}$ is  a filter in ${\mathscr B}$. To show that ${\mathscr F}$ is an ultrafilter, let $B\in {\mathscr B}$ be such that $B\cap F$ is non-empty for each $F\in {\mathscr F}$. If  $B\notin {\mathscr F}$ then $p\notin\mathrm{cl}_{\beta X} B$. Thus $p\in\mathrm{cl}_{\beta X}(X\backslash  B)$, i.e., $X\backslash  B\in {\mathscr F}$. But this is not possible, as $X\backslash  B$ misses $B$. Therefore $B\in {\mathscr F}$. To show that ${\mathscr F}$ is free, let $x\in X$. Then (since $p\notin X$) there exist some disjoint open neighborhoods $U$ and $V$ of $p$ and $x$ in $\beta X$, respectively. Then \[p\in\mathrm{cl}_{\beta X} U=\mathrm{cl}_{\beta X} (U\cap X)\]
and thus (since $U\cap X\in {\mathscr O}\subseteq{\mathscr B}$) we have $x\notin U\cap X\in {\mathscr F}$. Therefore $x\notin\bigcap{\mathscr F}$. Thus $\bigcap{\mathscr F}=\emptyset$ and  ${\mathscr F}$ is free. To show that  ${\mathscr F}$ has c.i.p., let $F_1,F_2,\ldots\in {\mathscr F}$. Then $p\in\mathrm{cl}_{\beta X} F_n$ for each $n\in\mathbb{N}$ and thus, since by Lemma \ref{JEUDC} we have
\[\mathrm{cl}_{\beta X}F_n\backslash X=\mathrm{Ex}_X(\mathrm{int}_XF_n)\backslash X\]
it follows that $p\in\mathrm{Ex}_X(\mathrm{int}_XF_n)$. For each $n\in\mathbb{N}$, let $f_n:\beta X\rightarrow[0,1]$ be continuous and such that
\[f_n(p)=0\mbox{ and }f_n|\big(\beta X\backslash\mathrm{Ex}_X(\mathrm{int}_XF_n)\big)\equiv 1.\]
Then
\[p\in Z=\bigcap_{n=1}^\infty\mathrm{Z}(f_n)\in {\mathscr Z}(\beta X).\]
If $Z\cap X=\emptyset$ then $Z\subseteq \beta X\backslash\upsilon X$, which is a contradiction, as $p\in\upsilon X$.  Thus, using Lemma \ref{TYU} we have
\begin{eqnarray*}
\emptyset\neq Z\cap X&=&\bigcap_{n=1}^\infty\mathrm{Z}(f_n)\cap X\\&\subseteq&\bigcap_{n=1}^\infty\mathrm{Ex}_X(\mathrm{int}_XF_n) \cap X=\bigcap_{n=1}^\infty\mathrm{int}_XF_n \subseteq\bigcap_{n=1}^\infty F_n.
\end{eqnarray*}
Therefore
\[\bigcap_{n=1}^\infty F_n\neq\emptyset.\]
This  show that  ${\mathscr F}$ has c.i.p.
\end{proof}

\begin{remark}
In the case when $X$ is normal, using Lemma \ref{FRAYU}, one can give an alternative proof for Theorem \ref{JUDC}. To show this assume that  there exists a non-trivial $\{0,1\}$-valued Baire measure $\nu$ on $X$ whose support is empty. Let
\[{\mathscr F}=\big\{B\in {\mathscr B}:\nu^*(X\backslash B)=0\big\}\]
in which $\nu^*$ is the outer measure induced by $\nu$.  We verify that ${\mathscr F}$ is a free ultrafilter in ${\mathscr B}$ with c.i.p. Obviously, ${\mathscr F}$ is non-empty (as $X\in {\mathscr F}$) and $\emptyset\notin {\mathscr F}$ (as $\nu$ is non-trivial). Note that for any $F,G\in{\mathscr F}$ since
\[\nu^*\big(X\backslash (F\cap G)\big)=\nu^*\big((X\backslash F)\cup (X\backslash G)\big)\leq\nu^*(X\backslash F)+\nu^*(X\backslash G)=0\]
we have $F\cap G\in {\mathscr F}$. Also, if $F\subseteq B$ for some $F\in{\mathscr F}$ and $B\in {\mathscr B}$ (since $X\backslash B\subseteq X\backslash F$) we have
\[\nu^*(X\backslash B)\leq\nu^*(X\backslash F)=0\]
and thus $B\in {\mathscr F}$. Therefore ${\mathscr F}$ is a filter in ${\mathscr B}$. To show that ${\mathscr F}$ is an ultrafilter, let $B\in {\mathscr B}$ be such that $B\cap F$ is non-empty for each $F\in {\mathscr F}$. Since $\mathrm{supp}(\nu)=\emptyset$, for each $x\in X$ there exists some $U_x\in\mathrm{Coz}(X)$ such that $x\in U_x$ and $\nu(U_x)=0$. By compactness of $\mathrm{bd}_XB$ there exist $x_1,\ldots,x_n\in X$ such that
\[\mathrm{bd}_XB\subseteq\bigcup_{i=1}^nU_{x_i}.\]
Let
\[U=\bigcup_{i=1}^nU_{x_i}\in\mathrm{Coz}(X).\]
Then
\[\mathrm{cl}_XB\backslash U\subseteq\mathrm{cl}_XB\backslash\mathrm{bd}_XB\subseteq\mathrm{int}_XB\]
and thus (since we are assuming that $X$ is normal) by the Urysohn Lemma there exists a continuous $f:X\rightarrow[0,1]$ such that
\[f|(X\backslash\mathrm{int}_XB)\equiv0\mbox{ and }f|(\mathrm{cl}_XB\backslash U)\equiv1.\]
Now, if $B\notin {\mathscr F}$ (by the definition of ${\mathscr F}$) we have $\nu^*(X\backslash B)=1$. Thus, since $X\backslash B\subseteq\mathrm{Z}(f)$ we have \[\nu\big(\mathrm{Z}(f)\big)=\nu^*\big(\mathrm{Z}(f)\big)=1\]
and therefore
\[\nu\big(\mathrm{Coz}(f)\big)=1-\nu\big(\mathrm{Z}(f)\big)=0.\]
Now, since
\[B\subseteq\mathrm{cl}_XB\subseteq(\mathrm{cl}_XB\backslash U)\cup U\subseteq\mathrm{Coz}(f)\cup U=\mathrm{Coz}(f)\cup \bigcup_{i=1}^nU_{x_i}\]
we have
\[\nu^*\big(X\backslash(X\backslash B)\big)=\nu^*(B)\leq\nu\big(\mathrm{Coz}(f)\big)+\sum_{i=1}^n\nu(U_{x_i})=0.\]
Therefore (by the definition of ${\mathscr F}$) $X\backslash B\in {\mathscr F}$, which is not possible, as it misses $B$. This contradiction shows that ${\mathscr F}$ is an ultrafilter. It remains to show that ${\mathscr F}$ has c.i.p. But this follows easily, as if $F_1,F_2,\ldots\in{\mathscr F}$, then since
\[\nu^*\Big(X\backslash\bigcap_{i=1}^\infty F_i\Big)=\nu^*\Big(\bigcup_{i=1}^\infty (X\backslash F_i)\Big)\leq\sum_{i=1}^\infty\nu^*(X\backslash F_i)=0\]
we have
\[\bigcap_{i=1}^\infty F_i\in{\mathscr F}.\]
Finally, note that ${\mathscr F}$ is free, as for each $x\in X$ since $\nu(U_x)=0$ (with $U_x$ as defined in the above) we have $X\backslash U_x\in{\mathscr F}$, and thus $x\notin \bigcap{\mathscr F}$.
\end{remark}

\begin{theorem}\label{JJKEDC}
Let $(Y, {\mathscr U},{\mathscr C})$ be a first-countable Hausdorff topological  measurable space. Let $(X, {\mathscr B})\subseteq(Y, {\mathscr C})$ and let $Y\backslash X$ be Lindel\"{o}f. If  there is no free ultrafilter in ${\mathscr B}$ with c.i.p. then there is no free ultrafilter in ${\mathscr C}$ with c.i.p.
\end{theorem}

\begin{proof}
For each $y\in Y$, let
\[\{V_n^y:n\in\mathbb{N}\}\]
be an open base at $y$ in $Y$. Suppose to the contrary that there exists a free ultrafilter ${\mathscr H}$ in ${\mathscr C}$ with c.i.p.

\medskip

\noindent {\bf Claim 1.} {\em For  each $y\in Y$ there exist some $n_y\in\mathbb{N}$ and $H_y\in{\mathscr H}$ such that $V_{n_y}^y\cap H_y=\emptyset$.}

\medskip

\noindent {\em Proof of Claim 1.} Suppose the contrary, i.e., suppose that for some $y\in Y$ the set $V_n^y\cap H$ is non-empty for each $n\in\mathbb{N}$ and $H\in {\mathscr H}$. Note that since ${\mathscr H}$ is an ultrafilter in ${\mathscr C}$ this implies that $V_n^y\in{\mathscr H}$ for each $n\in\mathbb{N}$. Since $Y$ is Hausdorff, we have
\[\bigcap_{n=1}^\infty V_n^y=\{y\},\]
and since ${\mathscr H}$ has c.i.p., we have
\[H\cap \{y\}=H\cap\bigcap_{n=1}^\infty V_n^y\neq\emptyset\]
i.e., $y\in H$ for each $H\in {\mathscr H}$, contradicting the fact that ${\mathscr H}$ is free. This shows Claim 1.

\medskip

\noindent {\bf Claim 2.} {\em $H\cap X$ is non-empty for  each $H\in{\mathscr H}$.}

\medskip

\noindent {\em Proof of Claim 2.} Suppose the contrary, i.e., suppose that $H\subseteq Y\backslash X$ for some $H\in {\mathscr H}$. Since
\[Y\backslash X\subseteq\bigcup \{V_{n_y}^y:y\in Y\backslash X\}\]
and $Y\backslash X$ is Lindel\"{o}f, we have
\[Y\backslash X\subseteq \bigcup_{i=1}^\infty V_{n_{y_i}}^{y_i}\]
for some  $y_1,y_2,\ldots\in Y\backslash X$. Now, by Claim 1 we have
\begin{eqnarray*}
H\cap\bigcap_{j=1}^\infty H_{y_j}&\subseteq&\Big(\bigcup_{i=1}^\infty V_{n_{y_i}}^{y_i}\Big)\cap\bigcap_{j=1}^\infty H_{y_j}\\&=&\bigcup_{i=1}^\infty \Big(V_{n_{y_i}}^{y_i}\cap\bigcap_{j=1}^\infty H_{y_j}\Big)\subseteq\bigcup_{i=1}^\infty (V_{n_{y_i}}^{y_i}\cap H_{y_i})=\emptyset
\end{eqnarray*}
contrary to the fact that ${\mathscr H}$ has c.i.p.

\medskip

\noindent Let
\[{\mathscr A}=\{H\cap X:H\in {\mathscr H}\}.\]
Then ${\mathscr A}\subseteq{\mathscr B}$ (as $(X, {\mathscr B})\subseteq(Y, {\mathscr C})$ and ${\mathscr H}\subseteq{\mathscr C}$) and by Claim 2 we have $\emptyset\notin {\mathscr A}$. Since ${\mathscr A}$ is obviously closed under finite intersections, as ${\mathscr H}$ is so, ${\mathscr A}$ is a filter-base in ${\mathscr B}$. Let ${\mathscr F}$ be an ultrafilter in  ${\mathscr B}$ such that ${\mathscr A}\subseteq{\mathscr F}$. Since ${\mathscr H}$ is free, we have
\[\bigcap{\mathscr F}\subseteq\bigcap{\mathscr A}=\bigcap{\mathscr H}\cap X=\emptyset\]
i.e., ${\mathscr F}$ also is free. To show that ${\mathscr F}$ has c.i.p., let $F_1,F_2,\ldots\in{\mathscr F}$. Let $F_n=C_n\cap X$ where $C_n\in{\mathscr C}$ for each $n\in\mathbb{N}$. For each $n\in\mathbb{N}$ and $H\in {\mathscr H}$, since $H\cap X\in {\mathscr A}\subseteq{\mathscr F}$, we have
\[\emptyset\neq H\cap X\cap F_n= H\cap X\cap C_n\subseteq H\cap C_n\]
and thus, since ${\mathscr H}$ is an ultrafilter in  ${\mathscr C}$, we have $C_n\in {\mathscr H}$. Now, for each $H\in{\mathscr H}$, since ${\mathscr H}$ has c.i.p., we have
\[H\cap \bigcap_{n=1}^\infty C_n\neq\emptyset,\]
and therefore
\[\bigcap_{n=1}^\infty C_n\in{\mathscr H}.\]
By Claim 2 we have
\[\bigcap_{n=1}^\infty F_n=\bigcap_{n=1}^\infty C_n\cap X\neq\emptyset.\]
This shows that ${\mathscr F}$ is a free ultrafilter in ${\mathscr B}$ with c.i.p., which is a contradiction.
\end{proof}

Note that in the above proof  we only need $Y$ to be first-countable at the points of $Y\backslash  X$.

\section{Examples of  measure spaces $(X,{\mathscr B},\mu)$ with an arbitrarily large number of free ultrafilters in ${\mathscr B}$ having c.i.p.}

\begin{example}\label{JJEDC}
Let $\zeta$ be a cardinal. Then there exists a measure space $(Z,{\mathscr D},\nu)$ having at least $\zeta$ free ultrafilters in ${\mathscr D}$ with c.i.p.
\end{example}

\begin{proof}
Let  $(X, {\mathscr B},\mu)$ be a measure space in which $\mu$ is a non-trivial $\{0,1\}$-valued measure which (is defined and) vanishes at singletons. Let $(Y,{\mathscr C},\lambda)$ be a $\sigma$-finite measure space in which singletons are measurable and such that $\mathrm{card}(Y)\geq\zeta$. By Proposition \ref{JJHHC} there exists a free ultrafilter ${\mathscr F}$ in ${\mathscr B}$ with  c.i.p. To see this, simply let
\[{\mathscr A}=\big\{\{x\}:x\in X\big\}\]
and observe that (since $\mu$ is non-trivial)
\[\mu\Big(\bigcup{\mathscr A}\Big)=\mu(X)=1\neq 0=\sup_{x\in X}\mu (x)=\sup_{A\in{\mathscr A}}\mu (A).\]
For each $y\in Y$, let ${\mathscr H}_y$ be an  ultrafilter in  ${\mathscr B}\times {\mathscr C}$ such that
\[\big\{F\times\{y\}:F\in{\mathscr F}\big\}\subseteq{\mathscr H}_y.\]
By the proof of Theorem \ref{JLFC} the ultrafilter ${\mathscr H}_y$, for each $y\in Y$, is free and has c.i.p. Note that ${\mathscr H}_y$'s are distinct if $y\in Y$ are distinct. The measure space $(X\times Y,{\mathscr B}\times{\mathscr C},\mu\times\lambda)$ has the desired property.
\end{proof}

\section{Questions}

We conclude this article with the following questions.

\begin{question}\label{IU}
In Theorem \ref{JEDC}, does the converse hold? More precisely, for a first-countable topological  measurable space $(X,{\mathscr O}, {\mathscr B})$ does the non-existence of any free ultrafilter in ${\mathscr B}$ with c.i.p. imply its realcompactness?
\end{question}

\begin{question}\label{SIU}
It is known that each finite measure space can be embedded in a perfect measure space. Even more, each finite measure space can be embedded in a compact (in the sense of  Marczewski) measure space. Find the corresponding choices of ${\mathbb U}$, $S_{\mathscr U}$ where ${\mathscr U}\in {\mathbb U}$, $(Z,{\mathscr D})$ and $\{{\mathscr D}_B:B\in {\mathscr B}\}$ in Theorem \ref{FDDDB} for any such embeddings.
\end{question}

\begin{center}
{\textbf{Acknowledgments}}
\end{center}

The author wishes to express his profound gratitude to the anonymous referee for his exceptionally careful reading of the article and his comments and suggestions which helped modifying the exposition of the article.

\end{document}